\documentclass[11pt,letterpaper]{amsart}

\usepackage[T1]{fontenc}
\usepackage[utf8]{inputenc}
\usepackage{amsmath,amssymb,amsthm,mathtools}
\usepackage{mathrsfs}
\usepackage{enumitem}
\usepackage{hyperref}
\usepackage{xcolor}

\makeatletter
\@namedef{subjclassname@2020}{\textup{2020} Mathematics Subject Classification}
\makeatother

\mathtoolsset{showonlyrefs=true}

\newtheorem{theorem}{Theorem}[section]

\newtheorem{proposition}[theorem]{Proposition}
\newtheorem{lemma}[theorem]{Lemma}

\theoremstyle{definition}

\theoremstyle{remark}
\newtheorem{remark}[theorem]{Remark}

\numberwithin{equation}{section}

\newcommand{\R}{\mathbb R}

\newcommand{\tr}{\operatorname{tr}}

\DeclareMathOperator{\diam}{diam}
\newcommand{\avg}[1]{\left\langle #1\right\rangle}
\newcommand{\norm}[1]{\left\|#1\right\|}

\begin{document}

\title{Sharp $L^q$ estimate of adjoint problem and Pucci conjecture}

\author[Hung]{Jui-Yun Hung}
\address{Mathematical Sciences Institute\\
The Australian National University\\
Canberra, ACT 2600, Australia} 
\email{JUI-YUN.HUNG@anu.edu.au}
\author[Wu]{Nan Wu}
\address{School of Mathematics \& IMS\\
Nanjing University\\
Nanjing, 210093, China} 
\email{nwu2@nju.edu.cn}

\begin{abstract}
Let $L$ be a linear second-order elliptic operator satisfying a nondegeneracy condition. We establish sharp $L^q$ estimates for solutions of the adjoint problem associated with $L$ and, by duality, derive sharp $L^p$ estimates for solutions of $Lu = f$. In particular, this confirms a conjecture of Carlo Pucci \cite{Pucci1966}.
\end{abstract}

\date{\today}

\maketitle

\section{Introduction}

Let $L$ be a linear elliptic second-order partial differential operator
of the form
\begin{equation}\label{eq:operator L}
    Lu=\mathcal A:D^2u=a^{ij}D_{ij}u
\end{equation}
in a bounded domain $\Omega\subseteq\R^n$, where
$\mathcal A=(a^{ij})$ is a measurable map from $\Omega$ into the space
of positive definite symmetric $n\times n$ matrices.
Throughout this paper, we write
\[
 A:B=\mathrm{Tr}(AB)
\]
for symmetric $n\times n$ matrices $A$ and $B$.

The classical Aleksandrov--Bakel'man maximum principle
\cite{Aleksandrov1961,Bakelman1961,GilbargTrudinger2001}
shows that, for any
$u\in C^2(\Omega)\cap C^0(\overline{\Omega})$ with
$u\le0$ on $\partial\Omega$, one has
\begin{equation}\label{ABP}
    \sup_{\Omega}u
    \le
    C(n,\Omega)
    \left\|
    \frac{(Lu)^-}{\det(\mathcal A)^{1/n}}
    \right\|_{L^n(\Omega)}.
\end{equation}

In \cite{Pucci1966}, Carlo Pucci made the following conjecture.
Assume that $L$ is uniformly elliptic in $\Omega$, normalized so that
\begin{equation}\label{eq:noramlization}
    \tr\mathcal A=1
    \qquad\text{and}\qquad
    \mathcal A\ge\delta I,
\end{equation}
for some positive constant $\delta\le1/n$.
Then, for any
$u\in C^2(\Omega)\cap C^0(\overline{\Omega})$ with
$u=0$ on $\partial\Omega$, one has
\begin{equation}\label{eq:Pucci main conj}
    \sup_{\Omega}|u|
    \le
    C\|Lu\|_{L^p(\Omega)},
\end{equation}
for some constant $C$ depending on $n$, $\delta$, $p$, and $\Omega$,
and for every exponent $p$ satisfying
\begin{equation}\label{eq:exponent restriction}
    p>p_*=\max\{n/2,n[1-(n-1)\delta]\}.
\end{equation}
By \cite{Pucci1966}, the exponent in
\eqref{eq:exponent restriction} is sharp.
The conjecture also holds when $\delta=1/n$, in which case
$L$ is the Laplacian.

The Pucci conjecture was proved in dimension two in
\cite{Astala2006}, for integer exponents in higher dimensions in
\cite{KuoTrudinger2007}, and for a broader range of exponents,
together with further refinements, in \cite{Trudinger2020}.

In this paper, we confirm the Pucci conjecture for the full range of
exponents in the following form.

\begin{theorem}\label{main}
Let $\Omega\subset\R^n$ be a bounded domain, and assume that
$L$ satisfies \eqref{eq:noramlization}.
If $p>p_*$ and
$u\in W^{2,p}_{\mathrm{loc}}(\Omega)\cap C(\overline\Omega)$,
then
\begin{equation}\label{eq: main esti}
 \sup_\Omega u
 \le
 \sup_{\partial\Omega}u
 +
 C(\diam\Omega)^{2-n/p}
 \norm{(Lu)^-}_{L^p(\Omega)},
\end{equation}
where $C$ is a constant depending only on $n$, $\delta$, and $p$.
In particular, the Pucci conjecture holds.
\end{theorem}

The proof of Theorem~\ref{main} is motivated by the
Fabes--Stroock approach \cite{FabesStroock1984} to higher
integrability of Green's functions.
In \cite{FabesStroock1984}, the authors studied nonnegative solutions
of adjoint equations $L^*v=0$ and applied the resulting estimates to
the difference $G_\Omega-G_{2B}$, which satisfies a homogeneous
adjoint equation in a suitable subdomain, in order to obtain higher
integrability estimates for the Green's function $G_\Omega$.
Here, $G_\Omega$ denotes the Green's function associated with $L$
on $\Omega$, and $G_{2B}$ denotes the Green's function associated
with $L$ on $2B\subset\Omega$.
See \cite{Fok1998,JungWoo2025} for further details and generalizations.

To treat the Pucci conjecture, we consider homogeneous adjoint
problems for $L$ on balls with suitable boundary conditions.
By establishing an $L^q$-estimate for solutions of these adjoint
problems, we obtain an $L^p$-estimate for solutions of the original
equation with zero boundary data on balls, where $p$ and $q$ are
H\"older conjugates.
A comparison argument then yields the desired estimate
\eqref{eq: main esti}.

\subsection*{Acknowledgments}

The authors would like to thank Neil Trudinger for bringing this
problem to their attention and for many helpful discussions and
encouragement. They also thank Qing Han for valuable comments and
suggestions. 
During the preparation of this manuscript, 
we became aware of two independent works. 
Chen, Trudinger, and Wang \cite{ChenTrudingerWang2026} studied 
the Pucci conjecture using a different method. 
Memana, Trudinger, and Yung \cite{MemanaTrudingerYung2026} developed a 
related two-dimensional theory 
involving fractional elementary symmetric functions and Hessian estimates.
The authors are grateful to the authors of
\cite{ChenTrudingerWang2026,MemanaTrudingerYung2026}
for sharing their preprints with them.

\subsection*{Disclosure on AI assistance}
The strategy for the proof of Theorem~\ref{main} 
was proposed by the authors, 
inspired by \cite{FabesStroock1984} and \cite{JungWoo2025}. 
In the course of proving Proposition~\ref{sharp moment esti}, 
the authors were initially able to establish the estimate only
for a restricted range of $q$. 
Through extensive interaction with AI-assisted tools 
(ChatGPT6 Astra), 
two ideas were suggested which is essential for removing 
this restriction: 
the introduction of an auxiliary parameter $t$ in the proof of 
Proposition~\ref{sharp moment esti},
and an approach based on the determinant and the $2$-Hessian 
that extends the exponent to the full range.

All theorem statements and proofs were written and verified by the authors, who take full responsibility for the contents of the paper.

\section{Proof of Theorem \ref{main}}
In this section, we will introduce the adjoint problem for 
$L$ in \eqref{eq:operator L} 
and complete the proof of Theorem \ref{main},
assuming the sharp $L^q$-estimates of adjoint solutions. 
Set $p>1$ and $q=p/(p-1)$ be its conjugate.
Throughout this paper, we always assume $n \ge 3$ and $0 <\delta < 1/n$,
as the Pucci conjecture is known for $n = 2$ and $\delta = 1/n$ corresponds
to the Laplace operator.
It is convenient to set
\begin{align}
\label{eq-a}
a = 1 - (n-1)\delta.
\end{align}

Let $x_0 \in \R^n$ and $r > 0$. Assume $f \in C^\infty(\overline{B}_r(x_0))$. Consider
\begin{align}
\label{eq:solutions-original}
\begin{cases}
a^{ij} \partial_{ij}v = -f &\text{in }B_r(x_0),\\
v = 0 &\text{on }\partial B_r(x_0),
\end{cases}
\end{align}
and the corresponding adjoint problem
\begin{align}
\label{eq:solutions-adjoint}
\qquad\,\,
\begin{cases}
\partial_i\partial_j(a^{ij}m_{r,x_0})=0 &\text{in }B_r(x_0),\\
m_{r,x_0} a^{ij} \nu_i \nu_j = 1 &\text{on }\partial B_r(x_0),
\end{cases}
\end{align}
where $\nu$ is the outward-pointing normal on $B_r(x_0)$.

We now present an important connection between \eqref{eq:solutions-original}
and \eqref{eq:solutions-adjoint}.

\begin{lemma}
\label{lemma-integral-formula-original-adjoint}
Assume $\mathcal A$ in \eqref{eq:operator L} is smooth on $\overline B_r(x_0)$.
Then, there exists a unique nonnegative smooth solution $m_{r,x_0}$ of 
\eqref{eq:solutions-adjoint} on $\overline B_r(x_0)$.
Moreover, for any $v \in C^\infty(\overline B_r(x_0))$ with $v = 0$ 
on $\partial B_r(x_0)$,
\begin{align}
\label{eq:integral-fromula}
\int_{B_r(x_0)}m_{r,x_0} Lv=\int_{\partial B_r(x_0)} v_\nu,
\end{align}
where $v_\nu$ is the normal derivative of $v$ with respect to the 
outward-pointing normal.
\end{lemma}

\begin{proof}
For simplicity, we write $B = B_r(x_0)$ and $m = m_{r,x_0}$.
We first prove the existence. The kernel of homogeneous adjoint operator
is trivial: if $h=0$ on the $\partial B$ and $\partial_i\partial_j(a^{ij}h)=0$, consider the equation $L\varphi=g$, $\varphi=0$ for an arbitrary smooth function $g$. Integration by parts gives $\int_Bhg=0$, hence $h=0$. Fredholm alternative therefore give the existence of $m$ \cite[Theorems~6.14--6.15]{GilbargTrudinger2001}.

Next, we prove \eqref{eq:integral-fromula}.
By a simple integration by parts, we have
\[
 \int_Bm a^{ij} \partial_{ij}v
 =\int_B\partial_i\partial_j(ma^{ij})\,v
  +\int_{\partial B}\bigl(ma^{ij}\partial_i v\nu_j
               -\partial_j(ma^{ij})v\nu_i\bigr).
\]
The interior term vanishes by the equation in \eqref{eq:solutions-adjoint}, 
and so does the second boundary term since $v=0$ on $\partial B$.
Since $v=0$ on the sphere, $Dv=v_\nu\nu$ there;
the remaining boundary term then reduces to
\[
\int_{\partial B}ma^{ij}\nu_i\nu_jv_\nu
=\int_{\partial B}v_\nu,
\]
where we used the boundary condition in \eqref{eq:solutions-adjoint}.
This proves \eqref{eq:integral-fromula}.

Finally, we establish the sign of $m$.
Fix any $f \in C^\infty_c(B)$ with $f \ge 0$.
Let $v$ be a solution of \eqref{eq:solutions-original}.
By the maximum principle, $v \ge 0$ in $B$. In particular, 
$v_\nu \le 0$ on $\partial B$.
Hence, by \eqref{eq:integral-fromula},
$$\int_B m f \ge 0.$$
This implies $m \ge 0$ in $B_r(x_0)$.
\end{proof}

\begin{remark}
By taking $v = (|x-x_0|^2-r^2)/2$,
a simple computation using \eqref{eq:integral-fromula} yields
\begin{equation}\label{eq-average-of-m}
\langle m_{r,x_0}\rangle=n,
\end{equation}
where $\langle\,\cdot\,\rangle = \frac{1}{|B_r(x_0)|} \int_{B_{r}(x_0)}$
is the integral average over $B_r(x_0)$.
\end{remark}

In the next result, we present the main $L^q$-estimate for adjoint problem in this paper.

\begin{proposition}\label{sharp moment esti}
Assume $\mathcal A$ in \eqref{eq:operator L} is smooth in $\overline B_r(x_0)$.
Then, there exists a unique nonnegative solution of
\eqref{eq:solutions-adjoint}
and, for any $n/(n-1)<q<\min\{2,na/(na-1)\}$ when $n\geq 4$, 
or $n/(n-1)<q<na/(na-1)$ when $n=3$, 
\begin{equation}\label{momentbound}
\langle m_{r,x_0}^q\rangle^{1/q}\le C_q,
\end{equation}
where $C_q$ is a positive constant depending only on $n$, $\delta$, and $q$.
\end{proposition}

We postpone the proof of Proposition \ref{sharp moment esti} to the next section. 
In the following, we will show how to apply this estimate to derive the Pucci's conjecture.

\begin{lemma}
Assume $\mathcal A$ in \eqref{eq:operator L} is smooth and
$f\in L^p(B_r(x_0))$ with $f \ge 0$.
Let $v_r$ and $v_{r/2}$ be solutions of \eqref{eq:solutions-original}
on $B_r(x_0)$ and $B_{r/2}(x_0)$, respectively.
Then, for any $p_* < p < n$, we have
\begin{equation}\label{ballestimate-1}
 v_r(x_0)\le C(n,\delta,q)r^{2-n/p}\norm f_{L^p(B_r(x_0))},
\end{equation}
for some constant $C(n,\delta,q)$ depending only on $n$, $\delta$, and $q$.
\end{lemma}
\begin{proof}
Without loss of generality, we assume $x_0=0$. 
We point out that for $q = p/(p-1)$, the assumption of 
Proposition \ref{sharp moment esti} is satisfied.
In fact, with $p_*$ as in \eqref{eq:exponent restriction}, 
$q = p/(p-1)$ satisfies
\[
 n/(n-1)<q< p_*/(p_*-1) = \min\{na/(na-1),n/(n-2)\}.
\]
For $n = 3$, the assumption holds trivially; and for $n \ge 4$, $n/(n-2)\le 2$,
which implies the desired assumption.
We first assume that $f$ is smooth near $\overline B_r$.
By the maximum principle, $v_s\ge 0$ for $s=r, r/2$,
and $v_r \ge v_{r/2}$ on $B_{r/2}$.
Hence, the difference $h=v_r-v_{r/2}$ satisfies $Lh=0$ in $B_{r/2}$ with 
$h \ge 0$ on $\partial B_{r/2}$. 
By the interior Harnack inequality 
\cite[Corollary~9.25]{GilbargTrudinger2001}, we have
\begin{align}
\label{eq:lower bound v_r}
 v_r\ge h\ge c(n,\delta)h(0)\quad\hbox{on }\partial B_{r/4}.
\end{align}

On $B_r\setminus\overline B_{r/4}$, define the function $\eta$ by
\[
 \eta(x)=\frac{(|x|/r)^{-(\delta^{-1}-2)}-1}{4^{\delta^{-1}-2}-1}.
\]
Note that $\eta\geq 0$ in $B_r\setminus\overline B_{r/4}$
with $\eta\equiv1$ on $\partial B_{r/4}$ and 
$\eta\equiv0$ on $\partial B_{r}$.
By \eqref{eq:noramlization}, it is straightforward to check that
\[
 L\eta\ge0.
\]
Indeed, the specific power is chosen to achieve this.
Now, by our assumption $f \ge 0$, $Lv_r=-f\le0 \le L \eta$ 
in $B_r\setminus\overline B_{r/4}$. 
On the other hand, by \eqref{eq:lower bound v_r} and 
the construction of $\eta$,
$v_r\ge c(n,\delta)h(0) = c(n,\delta)h(0)\eta$ on $\partial B_{r/4}$, 
and both vanish on $\partial B_{r}$.
The comparison principle then gives
\[
v_r\ge c(n,\delta)h(0)\eta \qquad \text{ in }B_r\setminus\overline B_{r/4}.
\]
In particular, the outward normal derivatives of $v_r$ 
and $c(n,\delta)h(0)\eta$ satisfy
\begin{equation}\label{eq: par vr}
 -\partial_\nu v_r\ge
 \frac{c(n,\delta)(\delta^{-1}-2)}{r(4^{\delta^{-1}-2}-1)}h(0) 
 \qquad\text{on }\partial B_r.
\end{equation}
Integrating \eqref{eq: par vr} on $\partial B_r$ and 
applying \eqref{eq:integral-fromula},
\[
0\leq h(0)\le C r^{2-n}\int_{\partial B_r} -\partial_\nu v_r=C r^{2-n}\int_{B_r} m_{r,0} f.
\]
By Proposition~\ref{sharp moment esti} and H\"older's inequality,
we obtain
\begin{equation}
\begin{aligned}
r^{2-n}\int_{B_r} m_{r,0} f  &\leq r^{2-n}\|m_{r,0}\|_{L^q(B_r)} \cdot\|f\|_{L^p(B_r)}\\
& = r^{2-n}|B_r|^{1/q}\langle m_{r,0}^q \rangle^{1/q}\cdot\|f\|_{L^p(B_r)}
\leq C r^{2-n/p} \|f\|_{L^p(B_r)},
\end{aligned}
\end{equation}
for some constant $C$ depending only on $n$, $\delta$, and $q$.
In sum, we obtain
\begin{equation}\label{eq: vr difference real}
\begin{aligned}
0\le v_r(0)-v_{r/2}(0)
 \le C(n,\delta, q)r^{2-n/p}\norm f_{L^p(B_r)}.
\end{aligned}
\end{equation}

We now iterate \eqref{eq: vr diffence of rj} to obtain \eqref{ballestimate-1}.
Set $r_j=2^{-j}r$. Applying \eqref{eq: vr difference real} on $B_{r_j}$,
we have
\begin{align}
\label{eq: vr diffence of rj}
v_{r_j}(0)-v_{r_{j+1}}(0)\le C(n,\delta, q)r_j^{2-n/p}\norm f_{L^p(B_{r})}.
\end{align}
We point out that $v_{r_j}(0) \rightarrow 0$ as $j \rightarrow \infty$.
In fact, the function
\[
 w_j(x)=\tfrac12|f|_{L^\infty(B_r)}
                  (r_j^2-|x|^2)
\]
satisfies $-Lw_j=|f|_{L^\infty(B_r)}$, because $\tr \mathcal A=1$.
Comparison principle then gives $0\le v_{r_j}(0)\le w_j(0)\to0$.
By \eqref{eq: vr diffence of rj}, $2 - n/p >0$, and a simple summation,
we then obtain \eqref{ballestimate-1} for smooth $f$ near $\overline B_r$.
For general $f \in L^p(B_r)$, we choose a nonnegative sequence 
$f_j \in C^\infty(\overline B_r)$ such that $f_j \rightarrow f$ in $L^p(B_r)$.
Let $v_j$ be a solution of \eqref{eq:solutions-original} 
with $f$ replaced by $f_j$.
The existence follows from classical $W^{2,p}$-theory \cite[Chapter 9]{GilbargTrudinger2001}.
Note that $p > n/2$. Then, elliptic estimates and the Sobolev embedding yield
$$|v_j - v|_{L^\infty (B_r)} \le C  \| f_j - f \|_{L^p(B_r)},$$
for some positive constant $C$.
Passing to the limit then proves \eqref{ballestimate-1} for $f \in L^p(B_r)$.
\end{proof}

\begin{proof}[Proof of Theorem~\ref{main}]
First, it suffices to consider $p_*<p<n$, the case $p\ge n$ follows from
classical Aleksandrov-Bakel’man maximum principle \eqref{ABP} and 
the H\"older's inequality. 

We first assume the coefficient $\mathcal A$ is smooth on $\R^n$
and $u \in C^\infty(\Omega)\cap C(\overline \Omega)$.
Denote by $f$ the trivial extension of $(Lu)^-$ to $\R^n$,
by setting the value to be zero outside $\Omega$.
Fix $x \in \Omega$.
Let $v$ be the solution of \eqref{eq:solutions-original} in $B_{2D}(x)$,
where $D = \mathrm{diam} \,\Omega$.
Then, $v \ge 0$ by the maximum principle, and $\Omega \subseteq B_{2D}(x)$.
Moreover, we have
$$L(u-v)=Lu+(Lu)^-\ge0 \text{ in } \Omega\quad \text{and}\quad 
u-v\le\sup_{\partial\Omega}u \text{ on }\partial\Omega.$$
Comparison principle then gives
$u(x)\le\sup_{\partial\Omega}u+v(x)$ and hence, by \eqref{ballestimate-1}, 
\begin{align*}
u(x)&\le\sup_{\partial\Omega}u+ C D^{2-n/p} \|f\|_{L^p(B_{2D}(x))} \\
&\le \sup_{\partial\Omega}u+
 C(\diam\Omega)^{2-n/p}
 \|(Lu)^-\|_{L^p(\Omega)} .
\end{align*}
This proves \eqref{eq: main esti} if $\mathcal A$ is smoothly 
defined on $\R^n$ and $u \in C^\infty(\Omega)$.
By a simple approximation, the estimate holds for $u \in W^{2,p}(\Omega)$
with the same constant.

We next consider general bounded measurable $\mathcal A$ defined on $\Omega$.
We extend $\mathcal A$ by $I/n$ outside $\Omega$ and mollify 
$\mathcal A$ by its convolution with the mollifier, 
i.e. $\mathcal A_\varepsilon=\rho_\varepsilon*\mathcal A$. 
A direct calculation then gives
\[
 \tr \mathcal A_\varepsilon=1,\qquad \mathcal A_\varepsilon\ge\delta I.
\]
In particular, \eqref{eq: main esti} holds for the operator
$L_{\mathcal A_\epsilon}: u \mapsto \mathcal A_\epsilon: D^2 u$.
Then, for $u \in W^{2,p}(\Omega)$, the dominated convergence theorem yields
\[
 \|(L_{A_\varepsilon}u)^--(Lu)^-\|_{L^p(\Omega)}
 \le \|(A_\varepsilon-A):D^2u\|_{L^p(\Omega)}
 \rightarrow0.
\]
Hence, the estimate \eqref{eq: main esti} can pass to bounded measurable 
$\mathcal A$ with the same constant, provided $u \in W^{2,p}(\Omega)$.
Finally, 
for general $u \in W^{2,p}_{\text{loc}} (\Omega) \cap C(\overline \Omega)$,
we apply the estimate on $\Omega_j = \{x \in \Omega;\,\mathrm{dist}(x,\partial \Omega)> 1/j\}$ for each $j \ge 1$.
Then, a simple approximation yields the desired estimate.
\end{proof}

\section{Sharp integral estimates of adjoint solution}
In this section, we give the proof of the Proposition \ref{sharp moment esti}. We will consider $q <2$ and $q \ge 2$ separately.
Throughout this section, we write
\begin{align}
\label{eq-C_q}
C_q=a^{-1}\Big[1-\frac{q(na-1)}{na}\Big]^{-1/q}.
\end{align}

We first consider the $n/(n-1)<q<\min\{2,na/(na-1)\}$.
This covers the case $n\ge 4$ in Proposition \ref{sharp moment esti}.
\begin{lemma}\label{sharp esti 1}
Let $x_0 \in \R^n$ and $r > 0$.
Assume $\mathcal A$ in \eqref{eq:operator L} is smooth on $\overline B_r(x_0)$.
Then, for any $n/(n-1)<q<\min\{2,na/(na-1)\}$,
\begin{equation}\label{eq:momentbound 2}
\langle m_{r,x_0}^q\rangle^{1/q}\le C_q,
\end{equation}
where $C_q$ is given as in \eqref{eq-C_q}.
\end{lemma}

\begin{proof}
Without loss of generality, we assume $B_r(x_0) = B_1$.
For simplicity, we write $m = m_{r,x_0}$.
We first derive an important estimate from the adjoint PDE.
The starting point is the following classical divergence identity 
for the Monge–Ampère operator 
\cite{TrudingerWang2008}. For any $C^3$-function $\psi$,
$$n \det D^2 \psi = \partial_i((\operatorname{cof} D^2 \psi)_{ij} \partial_j \psi).$$
Hence, for any $\psi \in C^3(\overline B_1)$ with $\psi = 0$ on $\partial B_1$,
the divergence theorem yields
\begin{equation}\label{eq:Z idnetity 2}
    \int_{B_1} \det D^2 \psi=\frac1n\int_{\partial B_1}\psi_\nu^n.
\end{equation}
We point out that the tangential block of $D^2 \psi$ on $\partial B_1$ is
$\psi_\nu I$.
We will use \eqref{eq:Z idnetity 2} to derive an estimate for a one-parameter
family of integrals.
Fix $t>0$.
By Lemma \ref{lemma-integral-formula-original-adjoint},
$\det (m\mathcal A + tI)^{1/(n-1)}$ is a smooth strictly positive function
in $\overline B_1$.
Then, the classical existence theory for Dirichlet problem of 
Monge–Ampère equation \cite[Theorem~3, p.~264, $k=n$]{CNS1985}, 
there exists a smooth solution $\psi$ of the equation
\begin{align}
\label{eq-MA-special-sol}
\det D^2\psi=\det (m\mathcal A + tI)^{1/(n-1)},\quad \psi|_{\partial B}=0,
 \quad D^2\psi>0.
\end{align}
By a simple estimate, we have
\begin{equation}\label{eq: T hessian psi}
\begin{aligned}
 (m\mathcal A + tI):D^2\psi&\ge n(\det (m\mathcal A + tI)\det D^2\psi)^{1/n}\\
 &=n\det (m\mathcal A + tI)^{1/(n-1)},
\end{aligned}
\end{equation}
where the middle equality uses $\det D^2\psi= \det (m\mathcal A + tI)^{1/(n-1)}$.
The choice of the power $1/(n-1)$ is to establish such an estimate.
By \eqref{eq:integral-fromula} and a simple integration by parts,
integrating \eqref{eq: T hessian psi} over $B_1$ yields
\begin{equation}\label{eq:Z idnetity 1}
 n\int_{B_1} \det (m\mathcal A + tI)^{1/(n-1)}\le\int_{B_1} (m\mathcal A + tI):D^2\psi=(1+t)\int_{\partial B_1}\psi_\nu.
\end{equation}
Note that $D^2 \psi > 0$ in $B_1$ and $\psi = 0$ on $\partial B_1$
imply $\partial_\nu \psi \ge 0$ on $\partial B_1$.
Then, the H\"older's inequality yields
$$n\int_{B_1} \det (m\mathcal A + tI)^{1/(n-1)}\le(1+t)|\partial B|^{(n-1)/n}\Big( \int_{\partial B_1} \psi_\nu^n\Big)^{1/n}.$$
By \eqref{eq:Z idnetity 2} and
the equation for $\psi$ in \eqref{eq-MA-special-sol},
the last term can be written as
\begin{align*}
 \Big( \int_{\partial B_1} \psi_\nu^n\Big)^{1/n} = \Big( n\int_{ B_1} \det (m\mathcal A + tI)^{1/(n-1)}\Big)^{1/n}.
\end{align*}
A simple calculation then yields
\begin{equation}\label{detbound}
 \big\langle\det(m\mathcal A+tI)^{1/(n-1)}\big\rangle\le(1+t)^{n/(n-1)}.
\end{equation}
We point out that, by letting $t\downarrow 0$, \eqref{detbound} yields
\[
 \big\langle\det(m\mathcal A)^{1/(n-1)}\big\rangle\le1.
\]
This gives an upper bound of $m^{n/(n-1)}$, since $\det \mathcal A\ge\delta^n$.
The purpose of introducing the auxiliary variable $t$ is to produce a 
family of inequalities from which other powers of $m$ can be extracted.

Let $\lambda_1,\ldots,\lambda_n$ be the eigenvalues of $\mathcal A$. By
\eqref{eq:noramlization},
$(\lambda_1,\ldots,\lambda_n)$ lies in the simplex
\[
 \mathcal S_\delta
 =
 \left\{(\lambda_1,\ldots,\lambda_n):
 \lambda_i\ge\delta,\ \sum_{i=1}^n\lambda_i=1\right\},
\]
whose vertices are the permutations of
$(a,\delta,\ldots,\delta)$, where $a$ is given by \eqref{eq-a}.
For fixed $t>0$ and $m\ge0$, the function
\[
 (\lambda_1,\ldots,\lambda_n)
 \longmapsto
 \sum_{i=1}^n\log(t+m\lambda_i)
\]
is concave on $\mathcal S_\delta$ and takes the same value at every
vertex. By Jensen's inequality, we obtain
\[
 \det(m\mathcal A+tI)
 \ge
 (t+am)(t+\delta m)^{n-1},
\]
and hence
\begin{equation}\label{scalarshift}
 \big\langle(t+am)^{1/(n-1)}(t+\delta m)\big\rangle
 \le (1+t)^{n/(n-1)}.
\end{equation}

We point out that \eqref{scalarshift} can be viewed 
as a scalar differential inequality. 
More precisely, set
\[
 F(t)=\langle(t+am)^{n/(n-1)}\rangle.
\]
By expressing
$$t + \delta m = \frac{\delta}{a} (t + am) + \frac{a-\delta}{a} t,$$
a straightforward computation, and differentiation under integral sign,
we can rewrite \eqref{scalarshift} as
\begin{equation}\label{eulerinequality}
\Big(\frac{\delta}{a} + \frac{na-1}{na}t \frac{d}{dt}\Big) F\le(1+t)^{n/(n-1)},\quad t>0.
\end{equation}
We now proceed formally. 
Note that $(s,t)\mapsto (s+t)^{n/(n-1)}$ is homogeneous of order $n/(n-1)$.
By a simple change of variable and interchanging the order of integration,
we have, formally,
$$\int_0^\infty F(t) t^{q-1 -n/(n-1)} dt = a^q\langle m^q\rangle \int_0^\infty (1+t)^{n/(n-1)} t^{q - 1 - n/(n-1)} dt. $$
This identity suggests that a weighted integral in the parameter $t$ 
can recover the $q$-th power of $m$. However, the integral on the right-hand
side diverges at infinity, since its integrand behaves like $t^{q-1}$.
We therefore need to subtract the leading terms in the large-$t$ expansion 
before carrying out the weighted integration.

To this end, we write
$$F(t) = t^{n/(n-1)} + \frac{n^2 a}{n-1} t^{1/(n-1)} + O(t^\frac{2-n}{n-1}),$$
where we used \eqref{eq-average-of-m} to compute $\langle m \rangle$.
The crucial observation is
$$\Big(\frac{\delta}{a} + \frac{na-1}{na}t \frac{d}{dt}\Big)
\Big(t^{n/(n-1)} + \frac{n^2 a}{n-1} t^{1/(n-1)}\Big) = t^{n/(n-1)} +\frac{n}{n-1} t^{1/(n-1)},$$
and the right-hand side matches exactly the asymptotics of $(1+t)^{n/(n-1)}$ near infinity
$$(1 + t)^{n/(n-1)} = t^{n/(n-1)} +\frac{n}{n-1} t^{1/(n-1)} + O(t^\frac{2-n}{n-1}).$$
Therefore, we define
\begin{equation}\label{eq: def of K(t)}
\begin{aligned}
G(t)&=F(t)-t^{n/(n-1)}-\frac{n^2 a}{n-1}t^{1/(n-1)}\\
&=\big\langle (t + ma)^{n/(n-1)} - t^{n/(n-1)} - \tfrac{n}{n-1} (ma) t^{1/(n-1)}\big\rangle.
\end{aligned}
\end{equation}
Then, \eqref{eulerinequality} implies, for any $t > 0$,
\begin{align}
\label{eulerinequality-1}
\Big(\frac{\delta}{a} + \frac{na-1}{na}t \frac{d}{dt}\Big) G\le(1+t)^{n/(n-1)} - t^{n/(n-1)} - \frac{n}{n-1} t^{1/(n-1)}.
\end{align}
The map $(s,t)\mapsto (t+s)^{n/(n-1)}-t^{n/(n-1)}-\frac{n}{n-1}st^{1/(n-1)}$
is homogeneous of degree $n/(n-1)$, and is nonnegative for any $s,t \ge 0$,
by the convexity of $t^{n/(n-1)}$.
Moreover, by Tonelli's theorem, the homogeneity,
and a simple change of variable,
\begin{align}
\label{eq-exact-identity-G-integral}
\int_0^\infty G(t) t^{q-1 -n/(n-1)} dt = a^q \langle m^q \rangle J,
\end{align}
where $J$ is given by
\begin{align}
\label{eq-integral-J}
J = \int_0^\infty \Big[(1+ t)^{n/(n-1)} - t^{n/(n-1)} - \frac{n}{n-1} t^{1/(n-1)} \Big] t^{q - 1 - \frac{n}{n-1}} dt.
\end{align}
Note that $0< J < \infty$ for $n/(n-1) < q < 2$.
On the other hand, multiplying $t^{q-1-n/(n-1)}$ and 
integrating \eqref{eulerinequality-1} over $(0,\infty)$,
a straightforward computation and integration by parts yield
\begin{equation}\label{eq: K J inequality}
\Big(1-\frac{q(na-1)}{na}\Big)
 \int_0^\infty G(t)t^{q-1 - n/(n-1)}\,dt\le J.
\end{equation}
We point out that, for $n/(n-1) < q < 2$, 
the boundary term vanishes when computing
$$\int_0^\infty t G'(t) t^{q-1 - n/(n-1)} dt 
= - \Big(q - \frac{n}{n-1} \Big) \int_0^\infty G(t) t^{q-1 - n/(n-1)} dt.$$
Combining \eqref{eq-exact-identity-G-integral} and \eqref{eq: K J inequality},
we obtain \eqref{eq:momentbound 2}.

\end{proof}

\begin{remark}
The restriction $q<2$ comes from the convergence of the 
integral in \eqref{eq-integral-J}.
We can try to subtract high-order terms in Taylor expansion
as in \eqref{eq: def of K(t)} to 
improve the integrability of \eqref{eq-integral-J} for $q>2$. 
However, this would introduce the unknown
quantity $\avg{m^2}$. There is no identity for this term
analogous to $\avg m=n$.
Thus, without additional
information on the adjoint solution, this renormalization procedure
cannot be continued directly.
\end{remark}

Next, we consider the remaining case in Proposition~\ref{sharp moment esti},
namely,
$$n=3\text{ and }2\le q<\frac{3a}{3a-1}.$$

We first prove the following technical lemma.
It can be viewed as a reversed Young's type inequality.

\begin{lemma}\label{lemma:Fp-duality}
Let $1<p\le2$ and $q=p/(p-1)$. Define
\begin{align}
\label{eq-F_p}
 F_p(w,\rho)
 =
 (w+\rho)^{p-1}\big[w-(p-1)\rho\big],
 \quad w,\rho\ge0,
\end{align}
and
\begin{align}
\label{eq-G_q}
 G_q(\tau,\eta)
 =
 (\tau+\eta)^{q-1}
 \big[\tau-(q-1)\eta\big],
 \quad \tau,\eta\ge0.
\end{align}
For every $\tau,\eta$ satisfying
\begin{equation}\label{eq:duality-cone}
 0 \le \eta\le(p-1)\tau,
\end{equation}
set
\begin{align}
\label{eq-w_-star(tau,eta)}
 w_*(\tau,\eta)=(\tau+\eta)^{q-2}
 \big[\tau-(q-2)\eta\big],
\end{align}
with $w_*(0,0) = 0$.
Then, $w_*(\tau,\eta) \ge 0$ and, for any $\rho\ge0$,
\begin{equation}\label{eq:Fp-duality}
 \frac1q G_q(\tau,\eta)
 \le
 \tau w_*(\tau,\eta)-\eta\rho-\frac1pF_p\big(w_*(\tau,\eta),\rho\big).
\end{equation}
\end{lemma}

\begin{proof}
We first verify that $w_*(\tau,\eta)\ge0$. Since
\[
 q-2=\frac{2-p}{p-1},
\]
the assumption $\eta\le(p-1)\tau$ gives
\[
 \tau-(q-2)\eta
 \ge
 \tau-(2-p)\tau
 =
 (p-1)\tau\ge0.
\]
Hence $w_*(\tau,\eta)\ge0$.

We next prove \eqref{eq:Fp-duality}.
If $\tau=\eta=0$, then $w_*(\tau,\eta)=0$,
$G_q(\tau,\eta)=0$, and $ F_p(0,\rho)=-(p-1)\rho^p$.
Hence, \eqref{eq:Fp-duality} holds trivially.
We now assume $\tau+\eta>0$.
In particular, \eqref{eq:duality-cone} implies $\tau > 0$,
and hence $w_*(\tau,\eta)> 0$.
The choice of $w_*$ arises from a {\it duality} between $F_p$ and $G_q$
with respect to the Lorentzian pairing
\[
 (\tau,\eta)\cdot(w,\rho)=\tau w-\eta\rho.
\]
To see this, a direct computation gives
\begin{align*}
\begin{cases}
\frac1p\partial_wF_p(w,\rho)
 =
 (w+\rho)^{p-2}\big[w+(2-p)\rho\big],\\[0.5em]
  -\frac1p\partial_\rho F_p(w,\rho)
 =
 (p-1)\rho(w+\rho)^{p-2}.
\end{cases}
\end{align*}
Consider the system
\begin{equation}\label{eq:dual-gradient-system}
\begin{cases}
\tau=\frac1p\partial_wF_p(w,\rho) = (w+\rho)^{p-2}\big[w+(2-p)\rho\big],\\[0.5em]
 \eta=-\frac1p\partial_\rho F_p(w,\rho) = (p-1)\rho(w+\rho)^{p-2}.
\end{cases}
\end{equation}
Adding the two equations in \eqref{eq:dual-gradient-system} gives
$\tau+\eta=(w+\rho)^{p-1}$,
and hence
$w+\rho=(\tau+\eta)^{q-1}$.
The second equation in \eqref{eq:dual-gradient-system} then gives
\begin{align}
\label{eq-duality-rho}
 \rho=
 \frac{\eta}{p-1}(\tau+\eta)^{q-2}.
\end{align}
Substituting this into the preceding identity yields
\[
 w=
 (\tau+\eta)^{q-2}
 \big[\tau-(q-2)\eta\big]
 =
 w_*(\tau,\eta).
\]
Thus, for $w_*$ given by \eqref{eq-w_-star(tau,eta)} and $\rho_* = \rho$
given by \eqref{eq-duality-rho},
$(w_*,\rho_*)$ solves \eqref{eq:dual-gradient-system}.
By a simple substitution, it is straightforward to check that
\[
 F_p(w_*,\rho_*) =G_q(\tau,\eta).
\]
Moreover, since $F_p$ is homogeneous of degree $p$, 
\eqref{eq:dual-gradient-system} implies
\[
 F_p(w_*,\rho_*)
 =
 \frac1p\Big[
 w_*\partial_wF_p(w_*,\rho_*)
 +\rho_*\partial_\rho F_p(w_*,\rho_*)
 \Big]
 =
 \tau w_*-\eta\rho_*.
\]
It follows that
\begin{equation}\label{eq:duality-critical-value}
 \tau w_*-\eta\rho_*
 -\frac1pF_p(w_*,\rho_*)
 =
 \frac1qG_q(\tau,\eta).
\end{equation}

Now, consider the function
\[
 \rho\mapsto
 \tau w_*-\eta\rho-\frac1pF_p(w_*,\rho).
\]
A direct computation gives
\[
 \partial_{\rho\rho}F_p(w,\rho)
 =
 -p(p-1)(w+\rho)^{p-3}
 \big[w+(p-1)\rho\big]\le0.
\]
Hence, the function above is convex in $\rho$. By
\eqref{eq:dual-gradient-system}, its derivative vanishes at
$\rho=\rho_*$, so $\rho_*$ is its minimum. Therefore, for every
$\rho\ge0$,
\[
 \tau w_*-\eta\rho-\frac1pF_p(w_*,\rho)
 \ge
 \tau w_*-\eta\rho_*-\frac1pF_p(w_*,\rho_*).
\]
Combining with \eqref{eq:duality-critical-value}, 
we obtain \eqref{eq:Fp-duality}.
\end{proof}

As in the proof of Lemma \ref{sharp esti 1}, 
we need to use an axillary comparison PDE.
Instead of using the Monge--Amp\`ere equation,
we use the $2$-Hessian equation in this case.

For a symmetric $3\times3$ matrix $H$, the $2$-Hessian is defined by
\[
 \sigma_2(H)
 =\frac12\big[(\tr H)^2-|H|^2\big],
\]
where $|H|^2=\tr(H^2)$. Denote the trace-free part of $H$ by
\[
 \mathring H=H-\frac13(\tr H)I.
\]
Since
\[
 |H|^2=\frac13(\tr H)^2+|\mathring H|^2,
\]
we have
\begin{equation}\label{eq:sigma2-tracefree}
 3\sigma_2(H)
 =(\tr H)^2-\frac32|\mathring H|^2.
\end{equation}
The admissible cone for the $2$-Hessian is
\[
 \Gamma_2^+
 =
 \{H:\tr H>0,\ \sigma_2(H)>0\}.
\]
In dimension three, \eqref{eq:sigma2-tracefree} gives the equivalent
description
\begin{align}
\label{eq-Gamma_2-cone-condition}
 \Gamma_2^+
 =
 \Big\{
 H:\tr H>\sqrt{\frac32}\,|\mathring H|
 \Big\}.
\end{align}
The operator $\sigma_2$ is elliptic on $\Gamma_2^+$.

We now prove the $L^q$-estimate in the remaining cases.
\begin{lemma}\label{sharp esti 2}
Let $x_0\in\R^3$ and $r>0$.
Assume $\mathcal A$ in \eqref{eq:operator L} is smooth on
$\overline B_r(x_0)$. Then, for any $2\le q<\frac{3a}{3a-1}$,
we have
\begin{equation}\label{eq:momentbound 3}
 \langle m_{r,x_0}^q\rangle^{1/q}
 \le
 C_q,
\end{equation}
where $C_q$ is given as in \eqref{eq-C_q}.
\end{lemma}

\begin{proof}
Without loss of generality, we assume $B_r(x_0)=B_1$.
For simplicity, write $m=m_{r,x_0}$ and let $p=\frac{q}{q-1}$.
Since $q\ge2$, we have $1<p\le2$. Moreover,
$q<\frac{3a}{3a-1}$ is equivalent to
\begin{equation}\label{eq:absorption condition}
 (3a-1)(q-1)<1.
\end{equation}

The proof consists of two steps.

\medskip
\textit{Step 1. We derive an estimate for the Hessian.}
We claim that, for any $\psi\in C^2(\overline B_1)$ satisfying
\begin{align}
\label{eq-condition-psi}
 \psi=0\quad\hbox{on }\partial B_1,
 \qquad
 \Delta\psi\ge0,
\end{align}
one has
\begin{equation}\label{eq:Fp-average}
 \big\langle\mathcal F_p(D^2\psi)\big\rangle
 \ge
 \langle\Delta\psi\rangle^p,
\end{equation}
where
\begin{align}
\label{eq-F_p-matrix}
 \mathcal F_p(H)
 =
 F_p\Big(
 \tr H,\sqrt{\frac32}\,|\mathring H|
 \Big),
\end{align}
for a symmetric matrix $H$ with $\tr H\ge0$,
and $F_p$ is given by \eqref{eq-F_p}.

Fix $\psi\in C^2(\overline B_1)$ satisfying
\eqref{eq-condition-psi}.
We start with the following integral identity for the $2$-Hessian,
analogous to \eqref{eq:Z idnetity 2}.
Recall the classical divergence identity for the $2$-Hessian
\cite{Reilly1973},
\[
 \operatorname{div}\big[(\Delta v I-D^2v)Dv\big]
 =
 2\sigma_2(D^2v).
\]
Hence, for any $v\in C^2(\overline B_1)$ with
$v=0$ on $\partial B_1$, the divergence theorem yields
\begin{equation}\label{newton2}
 \int_{B_1}\sigma_2(D^2v)
 =
 \int_{\partial B_1}v_\nu^2.
\end{equation}
We point out that $Dv=v_\nu\nu$ on $\partial B_1$, and the
tangential block of $D^2v$ is $v_\nu I$.

When $p=2$, by \eqref{eq:sigma2-tracefree},
\[
 \mathcal F_2(D^2\psi)
 =
 3\sigma_2(D^2\psi).
\]
Thus, \eqref{newton2} gives
\[
 \int_{B_1}\mathcal F_2(D^2\psi)
 =
 3\int_{\partial B_1}\psi_\nu^2.
\]
Since $\Delta\psi\ge0$ and $\psi=0$ on $\partial B_1$, we have
$\psi_\nu\ge0$. By H\"older's inequality and the divergence theorem,
\[
 3\int_{\partial B_1}\psi_\nu^2
 \ge
 \frac{3}{|\partial B_1|}
 \Big(\int_{\partial B_1}\psi_\nu\Big)^2
 =
 \frac{1}{|B_1|}
 \Big(\int_{B_1}\Delta\psi\Big)^2.
\]
Dividing by $|B_1|$ proves \eqref{eq:Fp-average} for $p=2$.

We now assume $1<p<2$. Define the truncated $2$-Hessian by
\[
 \mathcal Q(H)
 =
 \begin{cases}
  3\sigma_2(H)& \text{if }H\in\Gamma_2^+,\\
  0& \text{if }H\notin\Gamma_2^+.
 \end{cases}
\]
By \eqref{eq:sigma2-tracefree}, $\mathcal Q$ is continuous.

We first claim that, for every $v\in C^2(\overline B_1)$ with
$v=0$ on $\partial B_1$,
\begin{equation}\label{truncated esti}
 \int_{B_1}\mathcal Q(D^2v)
 \ge
 3\int_{\partial B_1}(v_\nu)_+^2.
\end{equation}
The proof is based on a comparison argument adapted from
\cite[Lemma~B, p.~265]{CNS1985}.
Choose smooth strictly positive functions
$h_j>\mathcal Q(D^2v)$ converging uniformly to
$\mathcal Q(D^2v)$ in $\overline B_1$.
By the classical existence theory for the Dirichlet problem for
the Hessian equation
\cite[Theorem~3, p.~264, $k=2$]{CNS1985}, there exists a smooth
function $v_j$ such that
\[
 3\sigma_2(D^2v_j)=h_j,\qquad
 D^2v_j\in\Gamma_2^+,\qquad
 v_j=0\quad\hbox{on }\partial B_1.
\]
If $v_j-v$ had a positive interior maximum, then at that point
$D^2v\ge D^2v_j$.
Since $\Gamma_2^+$ is monotone under addition of nonnegative matrices
and $\sigma_2$ is elliptic on $\Gamma_2^+$, this would imply
\[
 \mathcal Q(D^2v)
 =
 3\sigma_2(D^2v)
 \ge
 3\sigma_2(D^2v_j)
 =
 h_j,
\]
a contradiction. Hence $v_j\le v$ in $B_1$.

Moreover, $\Delta v_j>0$ and $v_j=0$ on $\partial B_1$, so
$v_j\le0$. It follows that
$(v_j)_\nu\ge v_\nu$ and $(v_j)_\nu\ge0$ on $\partial B_1$, and hence
\[
 (v_j)_\nu\ge(v_\nu)_+.
\]
Applying \eqref{newton2} to $v_j$ gives
\[
 \int_{B_1}h_j
 =
 3\int_{\partial B_1}(v_j)_\nu^2
 \ge
 3\int_{\partial B_1}(v_\nu)_+^2.
\]
Letting $j\to\infty$, we obtain \eqref{truncated esti}.

We now introduce a one-parameter family as in the proof of
Lemma~\ref{sharp esti 1}.
For each $t>0$, applying \eqref{newton2} and
\eqref{truncated esti} to
$\pm\Big[\frac{t}{6}(|x|^2-1)-\psi\Big]$
yields
\begin{equation}\label{eq:shifted-truncated}
\begin{aligned}
\int_{B_1}
 \Big[
 3\sigma_2\Big(D^2\psi-\frac{t}{3}I\Big)
 -
 \mathcal Q\Big(\frac{t}{3}I-D^2\psi\Big)
 \Big]  \le
 3\int_{\partial B_1}
 \Big(\psi_\nu-\frac{t}{3}\Big)_+^2,
\end{aligned}
\end{equation}
where we used
$(a_+)^2=a^2-(-a)_+^2$ to combine the two boundary integrals.

Replacing $v$ by $\psi$ in \eqref{newton2} and subtracting the
resulting identity from \eqref{eq:shifted-truncated}, we obtain
\[
\begin{aligned}
&\int_{B_1}
 \Big[
 3\sigma_2\Big(D^2\psi-\frac{t}{3}I\Big)
 -\mathcal Q\Big(\frac{t}{3}I-D^2\psi\Big)
 -3\sigma_2(D^2\psi)
 \Big]  \\
&\qquad\le
 3\int_{\partial B_1}
 \Big[
 \Big(\psi_\nu-\frac{t}{3}\Big)_+^2
 -\psi_\nu^2
 \Big].
\end{aligned}
\]
By the elementary identity
\[
 3\sigma_2(H)
 -
 3\sigma_2\Big(\frac{t}{3}I-H\Big)
 =
 2t\tr H-t^2,
\]
we then obtain
\begin{equation}\label{shiftcomparison}
\begin{aligned}
\int_{B_1}
 \Big[
 2t\Delta\psi-t^2
 +\mathcal Q\Big(\frac{t}{3}I-D^2\psi\Big)
 \Big]  \ge
 3\int_{\partial B_1}
 \Big[
 \psi_\nu^2
 -
 \Big(\psi_\nu-\frac{t}{3}\Big)_+^2
 \Big].
\end{aligned}
\end{equation}

We now extract the exponent $p$ from this one-parameter family.
For a symmetric matrix $H$ with $\tr H\ge0$, the trace of
$tI/3-H$ is $t-\tr H$, while its trace-free part is
$-\mathring H$. Hence, by \eqref{eq-Gamma_2-cone-condition},
\[
 \frac{t}{3}I-H\in\Gamma_2^+
 \quad\Longleftrightarrow\quad
 t>\tr H+\sqrt{\frac32}\,|\mathring H|.
\]
Therefore, it is straightforward to check that
\[
 2t\tr H-t^2+
 \mathcal Q\Big(\frac{t}{3}I-H\Big)
 =
 \begin{cases}
  2t\tr H-t^2,
  &
  0<t<\tr H+\sqrt{\frac32}\,|\mathring H|,
  \\[2mm]
  (\tr H)^2-\frac32|\mathring H|^2,
  &
  t\ge\tr H+\sqrt{\frac32}\,|\mathring H|.
 \end{cases}
\]
A direct calculation gives
\begin{equation}\label{eq:mellin-Fp}
 \int_0^\infty
 \Big[
 2t\tr H-t^2
 +\mathcal Q\Big(\frac{t}{3}I-H\Big)
 \Big]t^{p-3}\,dt
 =
 \frac{2}{p(p-1)(2-p)}\mathcal F_p(H).
\end{equation}
The integral is convergent for $1<p<2$: near $t=0$ the
integrand is $O(t^{p-2})$, while for large $t$ it is
$O(t^{p-3})$.

On the other hand, for every $j\ge0$, a direct calculation gives
\[
 3\int_0^\infty
 \Big[
 j^2-\Big(j-\frac{t}{3}\Big)_+^2
 \Big]t^{p-3}\,dt
 =
 \frac{2\,3^{p-1}}{p(p-1)(2-p)}j^p.
\]
Since $\Delta\psi\ge0$ and $\psi=0$ on $\partial B_1$, we have
$\psi\le0$ and $\psi_\nu\ge0$.
Multiplying \eqref{shiftcomparison} by $t^{p-3}$,
integrating over $(0,\infty)$, and using Fubini's theorem and
\eqref{eq:mellin-Fp}, we obtain
\begin{equation}\label{eq:Fp-boundary}
 \int_{B_1}\mathcal F_p(D^2\psi)
 \ge
 3^{p-1}\int_{\partial B_1}\psi_\nu^p.
\end{equation}

By H\"older's inequality on $\partial B_1$,
\[
 \int_{\partial B_1}\psi_\nu^p
 \ge
 |\partial B_1|^{1-p}
 \Big(\int_{\partial B_1}\psi_\nu\Big)^p.
\]
Using
\[
 \int_{\partial B_1}\psi_\nu
 =
 \int_{B_1}\Delta\psi,
 \qquad
 |\partial B_1|=3|B_1|,
\]
we obtain \eqref{eq:Fp-average} for $1<p<2$.

\textit{Step 2. We apply Lemma~\ref{lemma:Fp-duality} to obtain the
$L^q$-estimate.}

We first relate the Lorentzian pairing in
Lemma~\ref{lemma:Fp-duality} to the matrix pairing.
Decomposing $m\mathcal A$ and $D^2\psi$ into their trace and
trace-free parts, the Cauchy--Schwarz inequality gives
\begin{equation}\label{eq:matrix-dual-pairing}
\begin{aligned}
 3m\mathcal A:D^2\psi
 &=
 m\Delta\psi
 +3\mathring{m\mathcal A}:\mathring{D^2\psi} \\
 &\ge
 m\Delta\psi
 -
 \sqrt6\,|\mathring{m\mathcal A}|
 \sqrt{\frac32}\,|\mathring{D^2\psi}|.
\end{aligned}
\end{equation}
In view of \eqref{eq-F_p-matrix} and \eqref{eq:matrix-dual-pairing}, 
the corresponding variables in Lemma \ref{lemma:Fp-duality} are
\[
 \tau=m,\quad
 \eta=\sqrt6\,|\mathring{m\mathcal A}|,
 \quad
 \rho=\sqrt{\frac32}\,|\mathring{D^2\psi}|,
 \quad w = \Delta \psi.
\]
In order to apply the estimate \eqref{eq:Fp-duality}, 
we need to verify the condition \eqref{eq:duality-cone}, namely, 
\begin{equation}\label{eq:dual-cone-mA}
 0\le
 \sqrt6\,|\mathring{m\mathcal A}|
 \le
 (p-1)m.
\end{equation}
Since $\tr\mathcal A=1$, the maximum of
$|\mathring{\mathcal A}|^2$ is attained at $(a,\delta,\delta)$ and its permutations. Hence,
\[
 |\mathring{\mathcal A}|^2
 \le
 \big(a-\tfrac13\big)^2
 +2\big(\delta-\tfrac13\big)^2
 =
 \frac{(3a-1)^2}{6}.
\]
In particular, we have
\begin{equation}\label{eq:tracefree-A}
 \sqrt6\,|\mathring{\mathcal A}|
 \le
 3a-1.
\end{equation}
Since $m\ge0$ and $\tr(m\mathcal A)=m$, it follows that
$\sqrt6\,|\mathring{m\mathcal A}|\le(3a-1)m.$
On the other hand, \eqref{eq:absorption condition} and
$q-1=1/(p-1)$ imply $3a-1<p-1$.
This proves \eqref{eq:dual-cone-mA}.
Motivated by \eqref{eq-w_-star(tau,eta)}, we choose $\psi$ to solve
\begin{equation}\label{eq:dual-Poisson}
\begin{cases}
 \displaystyle
 \Delta\psi
 =
 w_*\big(m,\sqrt6\,|\mathring{m\mathcal A}|\big)
 &\text{in }B_1,\\[1mm]
 \psi=0
 &\text{on }\partial B_1.
\end{cases}
\end{equation}
More explicitly,
\[
 \Delta\psi
 =
 \Big(
 m+\sqrt6\,|\mathring{m\mathcal A}|
 \Big)^{q-2}
 \Big[
 m-(q-2)\sqrt6\,|\mathring{m\mathcal A}|
 \Big].
\]
By \eqref{eq:dual-cone-mA} and Lemma~\ref{lemma:Fp-duality},
the right-hand side is nonnegative. Hence $\Delta\psi\ge0$, and
the estimate \eqref{eq:Fp-average} from Step 1 applies to $\psi$.

As in Step 1, it is convenient to write
\[
 \mathcal G_q(M)
 =
 G_q\big(
 \tr M,\sqrt6\,|\mathring M|
 \big),
\]
for a symmetric matrix $M$ with $\tr M\ge0$, where $G_q$ is
given by \eqref{eq-G_q}.
Applying Lemma~\ref{lemma:Fp-duality} pointwise with
\[
 \tau=m,\qquad
 \eta=\sqrt6\,|\mathring{m\mathcal A}|,
 \qquad
 \rho=\sqrt{\frac32}\,|\mathring{D^2\psi}|,
\]
and using \eqref{eq:dual-Poisson} and
\eqref{eq:matrix-dual-pairing}, we obtain
\begin{equation}\label{eq:matrix-duality}
 \frac1q\mathcal G_q(m\mathcal A)
 \le
 3m\mathcal A:D^2\psi
 -
 \frac1p\mathcal F_p(D^2\psi).
\end{equation}
Note that
$\big\langle m\mathcal A:D^2\psi\big\rangle = \langle\Delta\psi\rangle$,
by \eqref{eq:integral-fromula}.
Integrating \eqref{eq:matrix-duality}, \eqref{eq:Fp-average} yields
\[
 \frac1q
 \big\langle\mathcal G_q(m\mathcal A)\big\rangle
 \le
 3\langle\Delta\psi\rangle
 -
 \frac1p\langle\Delta\psi\rangle^p.
\]
By Young's inequality,
\[
 3\langle\Delta\psi\rangle
 -
 \frac1p\langle\Delta\psi\rangle^p
 \le
 \frac{3^q}{q}.
\]
Hence,
\begin{equation}\label{eq:Gq-average}
 \big\langle\mathcal G_q(m\mathcal A)\big\rangle
 \le
 3^q.
\end{equation}

Finally, we compare $\mathcal G_q(m\mathcal A)$ with $m^q$.
By the definition of $\mathcal G_q$, we have
\[
 \mathcal G_q(m\mathcal A)
 =
 m^q
 \big(1+\sqrt6\,|\mathring{\mathcal A}|\big)^{q-1}
 \big[1-(q-1)\sqrt6\,|\mathring{\mathcal A}|\big].
\]
By a simple differentiation, it is straightforward to check that the function
$x\mapsto(1+x)^{q-1}[1-(q-1)x]$
is decreasing for $x\ge0$.
Using \eqref{eq:tracefree-A}, we therefore obtain
\[
 \mathcal G_q(m\mathcal A)
 \ge
 (3a)^{q-1}
 [1-(q-1)(3a-1)]m^q.
\]
We point out that the coefficient on the right-hand side is positive
by \eqref{eq:absorption condition}. Combining this estimate with
\eqref{eq:Gq-average}, we obtain
\[
 (3a)^{q-1}
 [1-(q-1)(3a-1)]
 \langle m^q\rangle
 \le
 3^q.
\]
A simple rearrangement then yields \eqref{eq:momentbound 3}.
\end{proof}

\begin{remark}
The restriction $q < na/(na-1)$ appearing in Lemma \ref{sharp esti 1}
and Lemma \ref{sharp esti 2} is sharp, and 
the constant $C_q$ given by \eqref{eq-C_q} is a sharp $L^q$-upper bound
in \eqref{eq:momentbound 2} and \eqref{eq:momentbound 3}.
More precisely, there exists a bounded
measurable coefficient matrix $\mathcal A$ in $B_1$ 
satisfying \eqref{eq:noramlization}
and a weak solution $m$ of \eqref{eq:solutions-adjoint} such that,
for any $q<na/(na-1)$,
\begin{align}
\label{eq-sharp-L^q-m}
\langle m^q\rangle^{1/q}=C_q,
\end{align}
while $m\notin L^q(B_1)$
for any $q\ge na/(na-1)$.

Indeed, consider
\[
 \mathcal A(x)
 =
 a\frac{x\otimes x}{|x|^2}
 +
 \delta\Big(
 I-\frac{x\otimes x}{|x|^2}
 \Big),
 \qquad x\ne0.
\]
The eigenvalues of $\mathcal A$ are $(a,\delta,\ldots,\delta)$
and hence \eqref{eq:noramlization} holds.
Set
\[
 m(x)
 =
 \frac1a
 |x|^{-\frac{na-1}{a}}.
\]
It is straightforward to check that
$m$ is a weak solution of \eqref{eq:solutions-adjoint}.
Moreover, by a simple integration,
\[
\begin{aligned}
 \langle m^q\rangle
 &=
 n a^{-q}
 \int_0^1
 r^{n-1-\frac{q(na-1)}{a}}\,dr.
\end{aligned}
\]
Therefore, $m \in L^q(B_1)$ for $q < na/(na-1)$ 
and \eqref{eq-sharp-L^q-m} holds;
and $m \notin L^q(B_1)$ if $q \ge na/(na-1)$.
\end{remark}

\bibliographystyle{amsplain}
\bibliography{references}

\end{document}